\documentclass[11pt,reqno]{amsart}
\usepackage{geometry}                % See geometry.pdf to learn the layout options. There are lots.
\usepackage{graphicx}
\usepackage{amssymb}
\usepackage{epstopdf}
\usepackage{amsmath}   
\usepackage{amsthm}
\usepackage{amsfonts}
\usepackage{latexsym}
\usepackage{color}
\theoremstyle{plain}
\newtheorem{thm}{Theorem}

\newtheorem{lem}[thm]{Lemma}
\numberwithin{equation}{section}
\newtheorem{cor}[thm]{Corollary}

\theoremstyle{remark}

\theoremstyle{definition}
\newtheorem{defi}[thm]{Definition}

\newtheorem{rek}[thm]{Remark}

\newcommand{\RM}{\mathrm}
\newcommand{\bdisks}{\mathscr{D}\RM{isk}_n^{\partial}}
\newcommand{\BB}{\mathbb}

\usepackage{tikz-cd}

  \tikzset{commutative diagrams/.cd,
mysymbol/.style={start anchor=center,end anchor=center,draw=none}
}

\usepackage[mathscr]{euscript}

\makeatletter
\def\@xfootnote[#1]{%
  \protected@xdef\@thefnmark{#1}%
  \@footnotemark\@footnotetext}
\makeatother

\title{Stable splitting of mapping spaces via nonabelian Poincar{\'e} duality}
\author{Lauren Bandklayder }

\makeatletter
\def\l@subsection{\@tocline{2}{0pt}{2.5pc}{5pc}{}}
\def\l@subsubsection{\@tocline{2}{0pt}{4.5pc}{5pc}{}}
\makeatother

\begin{document}

\newpage
\maketitle

\begin{abstract}
We use nonabelian Poincar{\'e} duality to recover the classical stable splitting of compactly supported mapping spaces, $\RM{Map_c}(M,\Sigma^nX)$, where $M$ is a parallelizable $n$-manifold. Our method for deriving this splitting is new, and naturally extends to give a more general stable splitting of the space of compactly supported sections of a certain bundle on $M$ with fibers $\Sigma^nX$, twisted by the tangent bundle of $M$. This generalization incorporates possible $O(n)$-actions on $X$ as well as accommodating non-parallelizable manifolds. \end{abstract}

\tableofcontents 

\section{Introduction}
Both mapping spaces and configuration spaces play central roles in algebraic topology and, surprisingly, there is a beautiful connection between these two objects. Considering the space of all maps between a manifold and some arbitrary topological space is a daunting task. These mapping spaces are usually infinite dimensional and a priori may seem completely intractable. When specializing to the case where the target is suitably connective, the situation simplifies a little. In fact, these mapping spaces admit a filtration which is not obvious at first glance, and often after sufficient suspension, this filtration splits.  Consequently, one can study these spaces stably, and this allows us to understand this huge topological space in terms of tangible, well-behaved geometric pieces -- finite labelled configuration spaces. 

This connection has been studied extensively in the literature.  More formally, for $M$ a framed $n$-manifold and $X$ a connected topological space, there is a stable splitting of mapping spaces 
\begin{equation} \label{the splitting}
\Sigma^{\infty} \mathrm{Map_c}(M, \Sigma^nX) \simeq \underset{i\geq 1}{\bigvee} \Sigma^{\infty} \mathrm{Conf}_i(M, \partial M) \underset{\Sigma_i}{\wedge} X^{\wedge i},
\end{equation}
where $\RM{Map_c}(M,\Sigma^nX)$ is the space of compactly-supported maps from $M$ to $\Sigma^nX$ and $\RM{Conf}_i(M, \partial M)$ denotes the quotient of the configuration space of $i$ distinct points in $M$ by the subspace of configurations which contain at least one point in the boundary of $M$. Special cases of this result date back as far as the early seventies, including Milnor's splitting of $\Sigma \Omega \Sigma X$ \cite{Milnor}, Snaith's splitting of $\Omega^n \Sigma^nX$ \cite{Sn} which has since been revisited, for example, in \cite{Coh80, CMT, LMS, Vogt, Ksn1, Ksn2} and discussed equivariantly in \cite{Tillman}, and Goodwillie's splitting of the free loop space, $\Lambda \Sigma X$, proven in \cite{Coh87}. In 1975, McDuff used the scanning map to model $\mathrm{Map}_c(M,\Sigma^nX)$ by a certain labelled configuration space that comes with a filtration which can be observed to split \cite{McDuff}. Such splittings have more recently been revisited with techniques from Goodwillie calculus, for example in \cite{Arone, Kmaps1}. This is certainly not a complete history of this result, and for a beautiful exposition consolidating more of the history, references, proofs and examples of this splitting, the reader is encouraged to see the 1987 survey of B{\"o}digheimer \cite{Bod}.

Most earlier proofs of such a splitting proceed in the manner of McDuff: via an identification of $\mathrm{Map}_c(M,\Sigma^nX)$ with some particular model of $X$-labelled configurations in $M$ which \emph{only holds} when one is considering maps into an $n$-fold suspension. In our proof, we will begin, instead, with nonabelian Poincar{\'e} duality, a more general result which allows us to model maps into \emph{any} $n$-connective topological space. Such a model is constructed using the theory of factorization homology, also called topological chiral homology, as developed in \cite{AF, AFT, Lurie, Sal}. Nonabelian Poincar{\'e} duality, as articulated in \cite{Lurie, Sal, SegalPD} and generalized in \cite{AF, AFT}, tells us that for a parallelizable manifold, $M$, and $n$-connective topological space, $Z$, there is a homotopy equivalence $$\int_M \Omega^nZ \simeq \mathrm{Map_c}(M,Z)$$ between factorization homology with coefficients in the $n$-fold loop space $\Omega^nZ$ and the space of compactly supported maps from $M$ to $Z$. For more on nonabelian Poincar{\'e} duality, the interested reader may also see \cite{Klang, Miller}.

There is a further generalization of nonabelian Poincar{\'e} duality, which, in particular, accommodates non-parallelizable manifolds by replacing the mapping space, $\RM{Map_c}(M,Z)$, with an appropriate section space, see Theorem \ref{nonabe PD}. Using this more general statement in the case where the target is an $n$-fold suspension, $\Sigma^nX$, we will obtain the following generalization of the splitting (\ref{the splitting}):
\begin{thm} \label{main} For any $O(n)$-space, X, there is a stable splitting $$\Sigma^{\infty} \Gamma_{\mathrm{c}}(E) \simeq \underset{i\geq 1}{\bigvee} \Sigma^{\infty} \mathrm{Conf}^{\RM{fr}}_i(M, \partial M) \underset{\Sigma_i \ltimes O(n)^i}{\wedge} X^{\wedge i}.$$
When $M$ is parallelizable and $X$ given the trivial $O(n)$-action, this reduces to the classical splitting 
$$\Sigma^{\infty} \mathrm{Map_c}(M, \Sigma^nX) \simeq \underset{i\geq 1}{\bigvee} \Sigma^{\infty} \mathrm{Conf}_i(M, \partial M) \underset{\Sigma_i}{\wedge} X^{\wedge i}.$$ 
\end{thm} \noindent Roughly, $E\rightarrow M$ is certain bundle with fibers $\Sigma^n X$, twisted by the tangent bundle of $M$, and $\mathrm{Conf}^{\RM{fr}}_i(M, \partial M)$ can be thought of as a frame bundle on $\RM{Conf}_i(M, \partial M)$; see Theorem~\ref{nonabe PD} and Definitions \ref{rel conf} and \ref{framebdl} for a more rigorous discussion. We remark that this generalization is a special case of the generalization obtained in Theorem 4.8 of \cite{Tillman}, but our proof is quite different-- Manthorpe and Tilllman begin by replacing configuration spaces with their spaces of tubular neighborhoods, after which they appeal to scanning map techniques akin to those of McDuff, whereas we avoid such methods entirely.

The outline of our derivation of Theorem \ref{main} is as follows. We will appeal to nonabelian Poincar{\'e} duality to identify the space $\mathrm{\Gamma_c}(E)$ with the factorization homology of $M$ with coefficients in $\Omega^n \Sigma^nX$. Passing to spectra and using a variant of May's Approximation Theorem \cite{May}, we are able to invoke the Snaith splitting to reduce the proof of Theorem~\ref{main} to the computation of the factorization homology of the free $n$-disk algebra in spectra on $\Sigma^{\infty} X$. Finally, we will observe that this factorization homology splits, as exhibited by a straight-forward hypercover argument, given in full detail, for example, in [AF15a -- 5.5] and [AF14 -- 2.4.1]. 

\subsection*{Acknowledgements} I would like to thank John Francis for suggesting this project and for his exceptional guidance throughout its completion. I am also indebted to Ben Knudsen for very many helpful conversations and constructive suggestions on this work. Finally, this paper was written while partially supported by an NSF Graduate Research Fellowship, and I am very grateful for their support. 
\section{Factorization homology and nonabelian Poincar{\'e} duality} \label{prelims}

Factorization homology acts as a bridge between the algebraic and geometric study of manifolds. Though defined in far more general circumstances, for our purposes factorization homology is a device which takes as input an $n$-manifold, $M$ as well as an $n$-disk algebra in spaces, $A$, and returns an space denoted $\int_M A$. An $n$-disk algebra can be thought of as an $E_n$-algebra with the extra data of an $O(n)$-action which is compatible with the action on $E_n$ given by rotating disks. Alternatively, one can think of $n$-disk algebras as algebras for the semidirect product $E_n \rtimes O(n)$ in the sense of \cite{Wahl}. One can then think of the factorization homology $\int_M A$ as a configuration space of points in $M$ with labels in $A$, in which points are allowed to collide and, correspondingly, their labels interact according to the multiplication of $A$. With this heuristic, and given the classical connections between labelled configuration spaces and mapping spaces, it seems natural that there might be a corresponding relationship between factorization homology and mapping spaces. This relation is given by nonabelian Poincar{\'e} duality. We remark that as it is not our intention in this note to give a rigorous exposition on factorization homology,  the reader is strongly encouraged to see any of \cite{AF, AFT, AFT2, Lurie, Klang, BK, Sal} for more background.

We now recall a special case of nonabelian Poincar{\'e} duality for smooth $n$-manifolds with boundary. Let $Z$ be an $n$-connective, pointed, topological space and $M$ be a smooth $n$-manifold, possibly with boundary. Let $E_Z \rightarrow BO(n)$ be a fibration with pointed fiber $Z$ and which admits a section. Denote by $E_Z^{TM}$ the bundle over $M$ obtained by pulling back $E_Z$ along the tangent classifier $M \rightarrow BO(n)$. Note that as the fiber of a bundle over $BO(n)$, the space $Z$ inherits an action of $O(n)$ which we will use to regard $\Omega^nZ$ as an $n$-disk algebra. 
\begin{thm}[Nonabelian Poincar{\'e} duality] \label{nonabe PD}
There is a natural equivalence $$\int_M \Omega^n Z \simeq \Gamma_{\mathrm{c}}(E^{TM}_Z).$$
\end{thm}
For a proof of Theorem \ref{nonabe PD} in greater generality, see [AFT17a -- 3.18], and it may also be helpful to see [AFT17a -- 3.12, 3.16] for context.  The case we will consider is when $Z$ is equal to $\Sigma^nX$, with $X$ a connected, pointed space equipped with some $O(n)$-action. We will set $E_{\Sigma^nX}$ to be the Borel construction $EO(n) \times_{O(n)} \Sigma^nX,$ where we view $\Sigma^nX$ as an $O(n)$-space with the diagonal action, acting in the natural way on suspension coordinates and with the given action on $X$ coordinates, and the fibration $E_{\Sigma^nX} \rightarrow BO(n)$ is that which is induced by the fibration $EO(n) \rightarrow BO(n).$ Note that the section space $\Gamma_{\RM{c}}(E_{\Sigma^n X}^{TM})$ reduces to the section space of \cite{Bod, McDuff} when the $O(n)$-action on $X$ is trivial or $M$ is parallelizable.

\begin{rek} \label{rek2} Note that in order to make sense of the equivalence in Theorem \ref{nonabe PD}, one must interpret $\Omega^n Z$ as a $\mathscr{D}\mathrm{isk}_n^{\partial}$-algebra in the sense of [AFT17a -- 2.2, 2.8] so that this factorization homology is defined.  A $\mathscr{D}\mathrm{isk}_n^{\partial}$-algebra can be regarded as a triple $(A,B, \alpha)$ where $A$ is an $n$-disk algebra, $B$ an $(n-1)$-disk algebra, and $\alpha$ some kind of action of $A$ on $B$.  In the case where the $(n-1)$-disk algebra $B$ is trivial, this is simply the structure of an augmented $n$-disk algebra, so in particular, augmented $n$-disk algebras are examples of $\bdisks$-algebras. We will only consider the situation where $B$ is trivial and $A$ is an $n$-disk algebra of the form $\Omega^n Z$, which we will regard as an augmented $n$-disk algebra.
\end{rek}

\begin{rek} \label{parallel}
In the case where the manifold $M$ is framed, the section space of Theorem~\ref{nonabe PD} reduces to the mapping space $\mathrm{Map_c}(M;\Sigma^nX)$ because $E_{\Sigma^nX}$ is being pulled back via a null-homotopic map.\end{rek}

Our proof will also use a very small bit of the theory of zero-pointed manifolds developed in \cite{ZPM, PKD}.  Informally, a zero-pointed $n$-manifold is a pointed topological space that is an $n$-manifold away from the basepoint. The canonical example of a zero-pointed $n$-manifold is the quotient, $M/\partial M$, of an $n$-manifold by it's boundary, and in fact this is the only zero-pointed manifold we will ever consider. There is a notion of factorization homology of zero-pointed $n$-manifolds with coefficients in an augmented $n$-disk algebra,  discussed at length in section 3 of \cite{ZPM} as well as in \cite{PKD}. What is important to us is that for a manifold with boundary, $M$, the factorization homology with coefficients in an augmented $n$-disk algebra as defined in Section 2.2 of \cite{AFT} coincides with the factorization homology as defined in [AF14 -- 1.2.1] of the zero-pointed manifold $M/\partial M$ with coefficients in that same algebra. 

\begin{rek} We caution the reader that for a subspace $N \subset M$, we regard the quotient $M/N$ as the pushout of the diagram:
\[
\begin{tikzcd}
N \ar[hook]{r} \ar{d} &M \\
* &
\end{tikzcd}
\]
In particular, we adopt the convention that taking the quotient of a space by the empty set has the effect of adjoining a disjoint basepoint. 
\end{rek}

To conclude this section we will recall the definitions of the configuration spaces which are involved in Theorem~\ref{main}. 
\begin{defi}
Let $\mathrm{Conf}_i(M)$ denote the ordered configuration space of $i$ distinct points in $M$, topologized as a subspace of $M^i$. \end{defi} 
\noindent That is, $\mathrm{Conf}_i(M) \subset M^i$ is the subspace of those maps $\{1,...,i\}\rightarrow M$ which are injective.

\begin{defi} \label{rel conf}
Let $\mathrm{Conf}_i(M,\partial M)$ denote the quotient of $\mathrm{Conf_i}(M)$ by those configurations in which at least one point lies on the boundary of $M$.
\end{defi}

Because $\RM{Conf}_i(M)$ is an open submanifold of $M^i$, we can consider it's frame bundle which is an $O(ni)$-bundle canonically reducible to an $O(n)^i$-bundle. Denote it's reduction by $\RM{Fr}_{\RM{red}}(\RM{Conf}_i(M))$. We make the following definition, which is a special case of the definition of the frame bundle functor given in [AF14 -- 1.5.1]. 
\begin{defi} \label{framebdl} Define $\RM{Conf}_i^{\RM{fr}}(M, \partial M)$ to be the space obtained as the quotient of \\$\RM{Fr}_{\RM{red}}(\RM{Conf}_i(M))$ by it's restriction to the subspace of $\RM{Conf}_i(M)$ consisting of configurations such that at least one point lies in the boundary of $M$.
\end{defi}

\noindent With this definition $\RM{Conf}_i^{\RM{fr}}(M, \partial M)$ inherits the structure of an $O(n)^i$-space, equipped with a map to $\RM{Conf}_i(M, \partial M)$ which is a principal bundle away from the basepoint.  
\begin{rek} \label{frame}Note that the quotient $\RM{Conf}_i^{\RM{fr}}(M, \partial M) / O(n)^i$ is equivalent to $\RM{Conf}_i(M, \partial M)$ and this equivalence is $\Sigma_i$-equivariant with the evident $\Sigma_i$-actions.
\end{rek}

\section{May's approximation theorem and the Snaith splitting for $n$-disk algebras}

\vspace{3mm}
We will now give equivariant versions of May's Approximation Theorem and the Snaith splitting which we will use in our proof of Theorem~\ref{main}.

\begin{lem}[May's Approximation Theorem for augmented $n$-disk algebras] \label{Maylemma} For a connected, pointed topological space $X$, equipped with some $O(n)$-action, there is an equivalence of $n$-disk algebras:
$$\mathbb{F}^{\RM{aug}}_n(X) \xrightarrow{\simeq} \Omega^n\Sigma^nX$$
where $\BB{F}^{\RM{aug}}_n(X)$ denotes the free augmented $n$-disk algebra in pointed spaces on the $O(n)$-space $X$.
\end{lem}
\begin{proof} 
We will show there is an equivalence of $n$-disk algebras $\mathbb{F}_n(X) \xrightarrow{\simeq} \Omega^n\Sigma^nX$, where $\mathbb{F}_n(X)$ denotes the free $n$-disk algebra on $X$ without an augmentation. The augmented case then follows from the fact that any equivalence of $n$-disk algebras is an equivalence over the $n$-disk algebra given by a single point, and thus gives an equivalence of the corresponding augmented algebras.

We will use the universal property of free $n$-disk algebras to obtain an $n$-disk algebra map $\mathbb{F}_n(X) \rightarrow \Omega^n\Sigma^nX$. From here it will suffice to show the underlying map of $E_n$-algebras is an equivalence, but this is precisely May's original theorem \cite{May}. Recall that $\mathbb{F}_n(X)$ satisfies the universal property that for any $O(n)$ equivariant map $X \rightarrow A$, where $A$ is the underlying $O(n)$-space associated to an $n$-disk algebra, there exists a unique map of $n$-disk algebras $\BB{F}_n(X) \rightarrow A$ making the diagram below commute:
\[
\begin{tikzcd}
X \ar[hook]{d} \ar{r} & A \\
 \BB{F}_n(X) \ar[dashed]{ur}
\end{tikzcd}
\]

So all we need to see to finish our proof is that the natural map 
\begin{align*} X &\rightarrow \Omega^n\Sigma^nX 
\end{align*} is $O(n)$-equivariant. To see this, note the image under this map of an arbitrary element $x$ in $X$ is given by the pointed map 
\begin{align*}
S^n &\rightarrow S^n \wedge X \\
t &\mapsto [t,\RM{cst}_x(t)]
\end{align*} where $\RM{cst}_x$ denotes the constant map, $\RM{cst}_x(t)\equiv x$. 
Thus we simply need to see that for any element $g \in O(n)$, we have the equality $$[t,g \cdot x] =[g\cdot g^{-1}\cdot t,g\cdot \RM{cst}_x(g^{-1}\cdot t)].$$ Since we have the equality $$g\cdot \RM{cst_x}(g^{-1}\cdot t)=g\cdot x,$$ the claim immediately follows. 
\end{proof}

\noindent The reader may also see \cite{Wahl} for another proof of Lemma \ref{Maylemma} and generalization thereof to an equivariant version of May's Recognition Principle.

\begin{cor}[The Snaith splitting for augmented $n$-disk algebras] \label{Snaithcor} For a connected, pointed topological space $X$, equipped with some $O(n)$-action, there is an equivalence of augmented $n$-disk algebras
$$ \Sigma^{\infty}_+ \mathbb{F}^{\mathrm{aug}}_n(X) \simeq \mathbb{F}_n^{\RM{aug}}(\Sigma^{\infty}X)$$
where $\BB{F}_n^{\RM{aug}}(\Sigma^{\infty}X)$ denotes the free augmented $n$-disk algebra in spectra on $\Sigma^{\infty}X$.
\end{cor}
\begin{proof}
This follows from Lemma \ref{Maylemma} in exactly the same manner as the classical Snaith splitting follows from May's Approximation Theorem as proved, for example, in \cite{Coh80}.
\end{proof}

\section{The stable splitting of mapping spaces}

We are now ready to prove the main theorem: 
\begin{thm}[Stable splitting] For $M$ a smooth $n$-manifold and $X$ a connected, pointed topological space equipped with an action of $O(n),$ there is a stable splitting 
 $$\Sigma^{\infty} \Gamma_{\mathrm{c}}(E^{TM}_{\Sigma^nX}) \simeq \underset{i\geq 1}{\bigvee} \Sigma^{\infty} \mathrm{Conf}^{\RM{fr}}_i(M, \partial M) \underset{\Sigma_i \ltimes O(n)^i}{\wedge} X^{\wedge i}.$$
 When one takes the trivial $O(n)$-action, this reduces to the splitting 
  $$\Sigma^{\infty} \Gamma_{\mathrm{c}}(E^{TM}_{\Sigma^nX}) \simeq \underset{i\geq 1}{\bigvee} \Sigma^{\infty} \mathrm{Conf}_i(M, \partial M) \underset{\Sigma_i}{\wedge} X^{\wedge i},$$
and, in particular, when $M$ is parallelizable this gives the splitting 
$$\Sigma^{\infty} \mathrm{Map_c}(M, \Sigma^nX) \simeq \underset{i\geq 1}{\bigvee} \Sigma^{\infty} \mathrm{Conf}_i(M, \partial M) \underset{\Sigma_i}{\wedge} X^{\wedge i}.$$
\end{thm}

\begin{proof}
 The statement of nonabelian Poincar{\'e} duality (Theorem \ref{nonabe PD}) gives the following equivalence:
$$\int_{M}\Omega^n \Sigma^nX \xrightarrow{\simeq} \Gamma_{\mathrm{c}}(E^{TM}_{\Sigma^nX}) $$
where $E^{TM}_{\Sigma^nX}$ denotes the bundle constructed in the paragraph following Theorem \ref{nonabe PD} and where we regard $\Omega^n \Sigma^n X$ as an augmented $n$-disk algebra as in Remark \ref{rek2}. Now Lemma \ref{Maylemma} gives an equivalence of augmented $n$-disk algebras between $\Omega^n \Sigma^n X$ and $\mathbb{F}_n^{\RM{aug}}(X)$, the free augmented $n$-disk algebra on the pointed space $X$ with the specified $O(n)$-action. It follows that there is an equivalence   
$$ \int_{M}  \mathbb{F}_n^{\RM{aug}}(X) \xrightarrow{\simeq} \Gamma_{\mathrm{c}}(E^{TM}_{\Sigma^nX}).$$
We next apply the suspension spectrum functor, $\Sigma^{\infty}_+,$ to both sides of this equivalence. Because $\Sigma_+^{\infty}$ is a symmetric monoidal left adjoint, it commutes with factorization homology, see [AF15a -- 3.25]. This leads to the equivalence
$$ \int_{M}  \Sigma^{\infty}_+ \mathbb{F}_n^{\RM{aug}}(X) \xrightarrow{\simeq} \Sigma^{\infty}_+ \Gamma_{\mathrm{c}}(E^{TM}_{\Sigma^nX}),$$
which, by Corollary \ref{Snaithcor}, reduces to the equivalence
$$ \int_{M}  \mathbb{F}_n^{\RM{aug}}(\Sigma^{\infty} X) \xrightarrow{\simeq} \Sigma^{\infty}_+\Gamma_{\mathrm{c}}(E^{TM}_{\Sigma^nX}).$$ 
As mentioned in Section \ref{prelims}, this factorization homology, as defined in \cite{AFT}, coincides with the reduced factorization homology of the zero-pointed manifold {$M_*:=M/\partial M$}, as defined in \cite{PKD}. One can compute this factorization homology using a hypercover argument and, in fact, we have the equivalence
$$
\int_{M_*}  \mathbb{F}_n^{\RM{aug}}(\Sigma^{\infty} X) \simeq \underset{i \geq 0}{\bigvee} \mathrm{Conf}^{\RM{fr}}_i(M, \partial M) \underset{\Sigma_i \ltimes O(n)^i}{\wedge} (\Sigma^{\infty}X)^{\wedge i}, $$
see [AF14 -- 2.4.1] for full details of this computation, and [AF15a -- 5.5] for the computation in the case where $X$ is given the trivial $O(n)$-action.
Finally, taking the quotient of both sides by $S^0$ gives the desired splitting:
\begin{align*} \Sigma^{\infty}\Gamma_{\mathrm{c}}(E^{TM}_{\Sigma^nX}) \simeq \underset{i \geq1}{\bigvee} \mathrm{Conf}^{\RM{fr}}_i(M, \partial M) \underset{\Sigma_i \ltimes  O(n)^i}{\wedge} (\Sigma^{\infty} X)^{\wedge i} 
& \simeq \underset{i \geq 1}{\bigvee} \Sigma^{\infty} \mathrm{Conf}^{\RM{fr}}_i(M, \partial M) \underset{\Sigma_i \ltimes O(n)^i}{\wedge} X^{\wedge i}.
\end{align*}
The splitting of the section space $\Gamma_{\RM{c}}(E_{\Sigma^nX}^{TM})$ when $X$ has the trivial $O(n)$-action follows from Remark \ref{frame} and the splitting of the mapping space $\RM{Map_c}(M,\Sigma^nX)$ for a parallelizable $M$ follows from Remark \ref{parallel}.
\end{proof}

\end{document}